\documentclass[hyperref]{article}
\usepackage{graphicx} %% Package for Figure
\usepackage{float} %% Package for Float
\usepackage{amssymb}
\usepackage{amsmath}
\usepackage{mathtools}
\usepackage[thmmarks,amsmath]{ntheorem} %% If amsmath is applied, then amsma is necessary
\usepackage{bm} %% Bold Mathematical Symbols
\usepackage[colorlinks,linkcolor=cyan,citecolor=cyan]{hyperref}
\usepackage{extarrows}
\usepackage[hang,flushmargin]{footmisc} %% Let the footnote not indentation
\usepackage[square,comma,sort&compress,numbers]{natbib} %% Sort of References
\usepackage{mathrsfs} %% Swash letter
\usepackage[font=footnotesize,skip=0pt,textfont=rm,labelfont=rm]{caption,subcaption}
\usepackage{booktabs} %% tables with three lines
\usepackage{tocloft}
\usepackage[all]{xy}
\usepackage{enumerate}
\usepackage{sectsty}
\subsectionfont{\fontsize{10}{12}\selectfont}
{%% Environment of Proof
	\theoremstyle{nonumberplain}
	\theoremheaderfont{\bfseries}
	\theorembodyfont{\normalfont}
	\theoremsymbol{\mbox{$\Box$}}
	\newtheorem{proof}{Proof}
}

\usepackage{theorem}
\newtheorem{theorem}{Theorem}[section]
\newtheorem{lemma}[theorem]{Lemma}
\newtheorem{proposition}[theorem]{Proposition}
\newtheorem{definition}[theorem]{Definition}

{%% Environment of Remark
	\theoremheaderfont{\bfseries}
	\theorembodyfont{\normalfont}
	\newtheorem{remark}[theorem]{Remark}
}
\newtheorem{example}[theorem]{Example}
\graphicspath{{figure/}}                                 %% Path of Figures
\usepackage[a4paper]{geometry}                           %% Paper size
\begin{document}
	\title{\bf $n$-cotorsion pairs over formal triangular matrix rings}
	\date{}
	\author{\sffamily Taolue Long$^1$, Xiaoxiang Zhang$^{1,*}$\\
		{\sffamily\small $^1$ School of Mathematics, Southeast University
			Nanjing 210096, P. R. China }
		%	\\	{\sffamily\small $^2$ The second Institution }
	}
	\renewcommand{\thefootnote}{\fnsymbol{footnote}}
	\footnotetext[1]{Corresponding author.
		
		E-mail:~tllong6688@163.com,~z990303@seu.edu.cn}
	\maketitle
	
	%MS+++++++++++++++++++++ Abstract +++++++++++++++++++++++++
	{\noindent\small{\bf Abstract:}
Let $\Lambda=\begin{pmatrix}
    A & 0\\
    U & B
\end{pmatrix}$ be a formal triangular matrix ring where $A,B$ are rings and $U$ is a $(B,A)$-bimodule.
        In this paper, we study some special classes over the formal triangular matrix ring $\Lambda$. Further, using these special classes, we construct a left (resp. right) $n$-cotorsion pair over the formal triangular matrix ring $\Lambda$ from left (resp. right) $n$-cotorsion pairs over $A$ and $B$.
        Finally, we give an example to illustrate our main result.
	}
	
	\vspace{1ex}
	
	
	%MS++++++++++++++++++++++++++++++ Main body ++++++++++++++++++++
	\section{Introduction}
Cotorsion pairs were introduced by Salce in the category of abelian groups in \cite{Salce1979}.
Later, the concept of cotorsion pairs was  generalized to abelian category, triangulated category and exact category (\cite{Gillespie2016,Nakaoka2011,Nakaoka2018}).
Motivated by some properties satisfied by Gorenstein projective and Gorenstein
injective modules over Iwanaga-Gorenstein rings, Huerta, Mendoza and P\'{e}rez \cite{HMP2021} present the concept of left and right $n$-cotorsion pairs in an abelian category.

Let $A$ and $B$ be rings and $U$ be a $(B,A)$-bimodule. The ring $\Lambda=\begin{pmatrix}
    A & 0\\
    U & B
\end{pmatrix}$ is known as a formal triangular
matrix ring with usual matrix addition and multiplication.
Formal triangular matrix rings play an important
role in both the ring theory and the representation theory of algebras.
This kind of rings is frequently employed to construct
significant examples and counterexamples, which make the theory of rings and modules more abundant and concrete.
Consequently, the properties of formal triangular matrix rings, as well as the modules over them, have attracted increasing attention and research interests \cite{JS2006,ARS,RPI1975,Green1982,AK1999,AK2000,KT2017}.
And Mao \cite{Mao2020} studies classical cotorsion pairs over formal triangular matrix rings.
Inspired by his work, we study (hereditary) left and right $n$-cotorsion pairs over formal triangular matrix rings in this paper.
And our main result is as follows:
\begin{theorem} \emph{(Theorem \ref{main result2})}
\noindent$(1)$ Let $(\mathcal{C},\mathcal{D})$ be a hereditary right $n$-cotorsion pair in $A\operatorname{-Mod}$ and $(\mathcal{E},\mathcal{F})$ be a hereditary right $n$-cotorsion pair in $B\operatorname{-Mod}$.
If $\operatorname{Tor}^A_j(U,\mathcal{C})=0$ for $1\leq j\leq n+1$ and $(\mathfrak{P}^\mathcal{C}_\mathcal{E})^\vee_{n-1}$ is
closed under extensions, then $(\mathfrak{P}^\mathcal{C}_\mathcal{E},\mathfrak{A}^\mathcal{D}_\mathcal{F})$ is a hereditary right $n$-cotorsion pair in $\Lambda\operatorname{-Mod}$.

\noindent$(2)$ Let $(\mathcal{C},\mathcal{D})$ be a hereditary left $n$-cotorsion pair in $A\operatorname{-Mod}$ and $(\mathcal{E},\mathcal{F})$ be a hereditary left $n$-cotorsion pair in $B\operatorname{-Mod}$.
If $\operatorname{Ext}_B^j(U,\mathcal{F})=0$ for $1\leq j\leq n+1$ and $(\mathfrak{I}^\mathcal{D}_\mathcal{F})^\wedge_{n-1}$ is closed under extensions, then $(\mathfrak{A}^\mathcal{C}_\mathcal{E},\mathfrak{I}^\mathcal{D}_\mathcal{F})$ is a hereditary left $n$-cotorsion pair in in $\Lambda\operatorname{-Mod}$.
\end{theorem}

This paper is organized as follows.
In Section \ref{Section 2}, we give some terminologies and preliminary
results.
In Subsection \ref{Subsection 3.1}, we study some special classes over formal triangular matrix rings.
In Subsection \ref{Subsection 3.2}, using these special classes, we study left and right $n$-cotorsion pairs over formal triangular matrix rings (see Theorem \ref{main result1} and Thoerem \ref{main result3}).
In Subsection \ref{Subsection 3.3}, we study hereditary left and right $n$-cotorsion pairs over formal triangular matrix rings (see Theorem \ref{main result2}).
Finally, we give an example to illustrate our main result (see Example \ref{example}).

\section{Preliminaries}\label{Section 2}
	Throughout this paper, we always assume $n\geq 1$ and all rings are nonzero associative rings with identity.
	For a ring $R$, we denote by $R$-Mod (Mod-$R$) the category of left  (right) $R$-modules.
    Denote by $R\operatorname{-Proj}$ (resp. $R\operatorname{-Inj}$) the class of projective (resp. injective) left $R$-modules.
	The character module $\operatorname{Hom_{\mathbb{Z}}}(M,\mathbb{Q}/\mathbb{Z})$ of a module $M\in R\operatorname{-Mod}$ is denoted by $M^+$.
	Given a class $\mathcal{C}$ of $R\operatorname{-Mod}$, we denote $\mathcal{C}^{\perp_i}:=\{X\in R\operatorname{-Mod}\ |\ \operatorname{Ext}_R^i(\mathcal{C},X)=0\}$  and ${}^{\perp_i}\mathcal{C}:=\{X\in R\operatorname{-Mod}\ |\ \operatorname{Ext}_R^i(X,\mathcal{C})=0\}$ for $i\geq 1$.
	In addition, we define $\mathcal{C}^{\perp_{[1,m]}}:=\bigcap\limits_{i=1}^m\mathcal{C}^{\perp_i}$ and ${}^{\perp_{[1,m]}}\mathcal{C}:=\bigcap\limits_{i=1}^m{}^{\perp_i}\mathcal{C}$ for $m\geq 1$.
    In the sequel, we sometimes use the symbol $\mathcal{C}^{\perp_{[1,m]}}$ (resp. ${}^{\perp_{[1,m]}}\mathcal{C}$) to replace $\bigcap\limits_{i=1}^m\mathcal{C}^{\perp_i}$ (resp. $\bigcap\limits_{i=1}^m{}^{\perp_i}\mathcal{C}$).

\subsection{Formal triangular matrix rings}
In this paper, $\Lambda=\begin{pmatrix}
		A & 0\\ U & B
	\end{pmatrix}$
always means a formal triangular matrix ring, where $A$ and $B$ are rings and $U$ is a $(B,A)$-bimodule.
By \cite[Theorem 1.5]{Green1982}, the category $\Lambda\operatorname{-Mod}$ of left $\Lambda$-modules is equivalent to the category $\mathrm{H}$ whose objects are triples $M=\begin{pmatrix}
    M_1\\M_2
\end{pmatrix}_{\varphi^M}$, where $M_1\in A\operatorname{-Mod}, M_2\in B\operatorname{-Mod}$ and $\varphi^M:U\otimes_A M_1\to M_2$ is a left $B$-module homomorphism, and whose morphisms from $\begin{pmatrix}
    M_1\\M_2
\end{pmatrix}_{\varphi^M}$ to $\begin{pmatrix}
    N_1\\N_2
\end{pmatrix}_{\varphi^N}$ are pairs $\begin{pmatrix}
    f_1\\f_2
\end{pmatrix}$ such that $f_1\in\operatorname{Hom}_A(M_1,N_1), f_2\in\operatorname{Hom}_B(M_2,N_2)$ satisfying that
the following diagram is commutative.

$$\xymatrix{U\otimes_A M_1\ar[r]^{1\otimes f_1}\ar[d]_{\varphi^M} & U\otimes_A N_1\ar[d]^{\varphi^N}\\
M_2\ar[r]^{f_2} & N_2
}$$
Given a triple $M=\begin{pmatrix}
    M_1\\M_2
\end{pmatrix}_{\varphi^M}$ in $\mathrm{H}$, there exists an adjoint isomorphism $$\operatorname{Hom}_B(U\otimes_A M_1,M_2)\overset{\mathrm{F}}{\cong} \operatorname{Hom}_A(M_1,\operatorname{Hom}_B(U,M_2)).$$
We will denote by $\widetilde{\varphi^M}$ the morphism $\mathrm{F}(\varphi^M)\in \operatorname{Hom}_A(M_1,\operatorname{Hom}_B(U,M_2))$.

Analogously, the category $\operatorname{Mod-}\Lambda$ of right $\Lambda$-modules is equivalent to the category $\mathrm{G}$ whose objects are triples $W=(W_1,W_2)_{\varphi_W}$, where $W_1\in\operatorname{Mod-}A, W_2\in\operatorname{Mod-}B$ and $\varphi_W:W_2\otimes_B U\to W_1$ is a right $A$-module homomorphism, and whose morphisms from $(W_1,W_2)_{\varphi_W}$ to $(Q_1,Q_2)_{\varphi_Q}$ are pairs $(g_1,g_2)$ such that $g_1\in\operatorname{Hom}_A(W_1,Q_1), g_2\in\operatorname{Hom}_B(W_2,Q_2)$
and the following diagram is commutative.
$$\xymatrix{W_1 \ar[r]^{g_1} & Q_1\\
W_2\otimes_B U \ar[r]^{g_2\otimes 1}\ar[u]^{\varphi_W} & Q_2\otimes_B U \ar[u]_{\varphi_Q}
}
$$

In the rest of the paper, we will identify $\Lambda\operatorname{-Mod}$ (resp. $\operatorname{Mod-}\Lambda$) with the category $\mathrm{H}$ (resp. $\mathrm{G}$) and, whenever there is no possible confusion, we shall omit the morphism $\varphi^M$ (resp. $\varphi_W$).
Note that a sequence $0\to\begin{pmatrix}
    X_1\\X_2
\end{pmatrix}_{\varphi^X}\to\begin{pmatrix}
    Y_1\\Y_2
\end{pmatrix}_{\varphi^Y}\to\begin{pmatrix}
    Z_1\\Z_2
\end{pmatrix}_{\varphi^Z}\to 0$ in $\Lambda\operatorname{-Mod}$ is exact if and only if the two sequences $0\to X_1\to Y_1\to Z_1\to 0$ and $0\to X_2\to Y_2\to Z_2\to 0$ are exact.

\begin{lemma}\label{projective and injective}
Let $M=\begin{pmatrix}
    M_1\\M_2
\end{pmatrix}_f$ be a left $\Lambda$-module.

\noindent$(1)$\emph{\cite[Theorem 3.1]{AK2000}} $M$ is a projective left $\Lambda$-module if and only if $M_1$ is a projective left $A$-module, $\operatorname{Coker}f$ is a projective left $B$-module and $f$ is a monomorphism.

\noindent$(2)$\emph{\cite[Proposition 5.1]{AK1999}} and \emph{\cite[p. 956]{JS2006}} $M$ is an injective left $\Lambda$-module if and only if $\operatorname{Ker}\widetilde{f}$ is an injective left $A$-module, $M_2$ is an injective left $B$-module and $\widetilde{f}$ is an epimorphism.
\end{lemma}

\begin{lemma}\emph{\cite[Proposition 3.6.1]{KT2017}}\label{KT2017}
Let $M=\begin{pmatrix}
    M_1\\M_2
\end{pmatrix}_{\varphi^M}$ be a left $\Lambda$-module and $W=(W_1,W_2)_{\varphi_W}$ be a right $\Lambda$-module.
Then there is an isomorphism of abelian groups
$$W\otimes_\Lambda M\cong ((W_1\otimes_A M_1)\oplus(W_2\otimes_B M_2))/H,$$
where the subgroup $H$ is generated by all elements of the form $(\varphi_W(w_2\otimes u))\otimes x_1-w_2\otimes \varphi^M(u\otimes x_1)$ with $x_1\in M_1,w_2\in W_2,u\in U$.
\end{lemma}

	\begin{lemma}\emph{\cite[Lemma 3.2]{Mao2020}}\label{extension formula}
		Let $M=\begin{pmatrix}
			M_1\\ M_2
		\end{pmatrix}_{\varphi^M}$
		and
		$N=\begin{pmatrix}
			N_1\\N_2
		\end{pmatrix}_{\varphi^N}$
		be left $\Lambda$-modules.
		\begin{enumerate}[$(1)$]
			\item $\mathrm{Ext}_{\Lambda}^n(\begin{pmatrix}
				M_1\\ M_2
			\end{pmatrix},\begin{pmatrix}
				N_1\\0
			\end{pmatrix})\cong \mathrm{Ext}_A^n(M_1,N_1)$.
			
			\item $\mathrm{Ext}_{\Lambda}^n(\begin{pmatrix}
				0\\ M_2
			\end{pmatrix},\begin{pmatrix}
				N_1\\N_2
			\end{pmatrix})\cong \mathrm{Ext}_B^n(M_2,N_2)$.
			
			\item If $\mathrm{Tor}^A_i(U,M_1)=0$ for $1\leq i\leq n+1$,
			then $\mathrm{Ext}_{\Lambda}^n(\begin{pmatrix}
				M_1\\ U\otimes_A M_1
			\end{pmatrix},\begin{pmatrix}
				N_1\\N_2
			\end{pmatrix})\cong \mathrm{Ext}_A^n(M_1,N_1)$.
			
			\item If $\mathrm{Ext}^i_B(U,N_2)=0$ for $1\leq i\leq n+1$, then
			$\mathrm{Ext}_{\Lambda}^n(\begin{pmatrix}
				M_1\\ M_2
			\end{pmatrix},\begin{pmatrix}
				\mathrm{Hom}_B(U,N_2)\\N_2
			\end{pmatrix})\cong \mathrm{Ext}_A^n(M_2,N_2)$.
		\end{enumerate}
	\end{lemma}

	\begin{remark}
		With respect to Lemma \ref{extension formula}(3) (resp. (4)), Mao uses the condition $\mathrm{Tor}^A_i(U,M_1)=0$ (resp. $\mathrm{Ext}^i_B(U,N_2)=0$) for $i\geq1$ in the original paper \cite[Lemma 3.2]{Mao2020}.
		It is easy to check that the results still hold if the condition ``$i\geq 1$'' is replaced by ``$1\leq i\leq n+1$''.
		So we give a minor adjustment without proof.
	\end{remark}

\subsection{Resolution and coresolution dimensions}

Given a class $\mathcal{C}$ of $R\operatorname{-Mod}$, an object $M\in R\operatorname{-Mod}$ and a nonnegative integer $m\geq 0$,  a $\mathcal{C}$-\emph{resolution} of $M$ of length $m$ is an exact sequence
	$$0\to C_m\to C_{m-1}\to \cdots \to C_1\to C_0\to M\to 0 $$ in $R\operatorname{-Mod}$, where $C_i\in\mathcal{C}$ for every integer $0\leq i\leq m$.
	 The \emph{resolution dimension} of $M$ with respect to $\mathcal{C}$ (or the $\mathcal{C}$-\emph{resolution dimension} of $M$), denoted by $\text{resdim}_\mathcal{C}(M)$, is defined as the smallest nonnegative integer $m$ such that $M$ has a $\mathcal{C}$-resolution of length $m$.
	 If such $m$ does not exist, we set $\text{resdim}_\mathcal{C}(M):=\infty$.
	 Dually, we have the concepts of $\mathcal{C}$-\emph{coresolutions} of $M$ of length $m$ and of \emph{coresolution dimension} of $M$ with respect to $\mathcal{C}$, denoted by $\text{coresdim}_\mathcal{C}(M)$. With respect to these two homological dimensions, we shall frequently consider the following classes of objects in $R\operatorname{-Mod}$:
	 $$\mathcal{C}^\wedge_m:=\{M\in R\operatorname{-Mod}:\text{resdim}_\mathcal{C}(M)\leq m\} \ \text{and} \ \mathcal{C}^\vee_m:=\{M\in R\operatorname{-Mod}:\text{coresdim}_\mathcal{C}(M)\leq m\}.$$

\subsection{Left (right) $n$-cotorsion pairs}
	\begin{definition} \rm \cite[Definition 2.1]{HMP2021}\label{n-cotorsion pair}
		Let $R$ be a ring and $\mathcal{C}$, $\mathcal{D}$ be classes in $R\operatorname{-Mod}$.
		
		\noindent(1) $(\mathcal{C},\mathcal{D})$ is called a \emph{left $n$-cotorsion pair} in $R\operatorname{-Mod}$ if the following conditions are satisfied:
		
		(a) $\mathcal{C}$ is closed under direct summands.
		
		(b) $\operatorname{Ext}^i_R(\mathcal{C},\mathcal{D})=0$ for $1\leq i\leq n$.
		
		(c) For every module $M\in R\operatorname{-Mod}$, there exists a short exact sequence $0\to D\to C \to M\to 0$ where $C\in\mathcal{C}$ and $D\in\mathcal{D}^\wedge_{n-1}$.
		
		\noindent(2) $(\mathcal{C},\mathcal{D})$ is called a \emph{right $n$-cotorsion pair} in $R\operatorname{-Mod}$ if the following conditions are satisfied:
		
		(a) $\mathcal{D}$ is closed under direct summands.
		
		(b) $\operatorname{Ext}^i_R(\mathcal{C},\mathcal{D})=0$ for $1\leq i\leq n$.
		
		(c) For every module $M\in R\operatorname{-Mod}$, there exists a short exact sequence $0\to M\to D \to C\to 0$ where $D\in\mathcal{D}$ and $C\in\mathcal{C}^\vee_{n-1}$.
	\end{definition}
	
	\begin{theorem}\emph{\cite[Theorem 2.7]{HMP2021}}\label{n-cotorsion pair2}
		Let $R$ be a ring and $\mathcal{C}$, $\mathcal{D}$ be classes in $R\operatorname{-Mod}$.
		
\noindent$(1)$ The following are equivalent:
		
		\emph{(a)} $(\mathcal{C},\mathcal{D})$ is a left n-cotorsion pair in $R\operatorname{-Mod}$.
		
		\emph{(b)} $\mathcal{C}=\bigcap\limits^n_{i=1}{}^{\perp_i}\mathcal{D}$ and for any $M\in R\operatorname{-Mod}$ there exists a short exact sequence $0\to D\to C \to M\to 0$ where $C\in\mathcal{C}$ and $D\in\mathcal{D}^\wedge_{n-1}$.
		
		\noindent$(2)$ The following are equivalent:
		
		\emph{(a)} $(\mathcal{C},\mathcal{D})$ is a right n-cotorsion pair in $R\operatorname{-Mod}$.
		
		\emph{(b)} $\mathcal{D}=\bigcap\limits^n_{i=1}\mathcal{C}^{\perp_i}$ and for any $M\in R\operatorname{-Mod}$ there exists a short exact sequence $0\to M\to D \to C\to 0$ where $D\in\mathcal{D}$ and $C\in\mathcal{C}^\vee_{n-1}$.
	\end{theorem}

\begin{proposition}\emph{\cite[Proposition 2.5]{HMP2021}}\label{prop2.3}
Let $R$ be a ring and $\mathcal{C}$, $\mathcal{D}$ be classes in $R\operatorname{-Mod}$.
If $\operatorname{Ext}^i_R(\mathcal{C},\mathcal{D})=0$ for $1\leq i\leq n$, then $\operatorname{Ext}^1_R(\mathcal{C},\mathcal{D}^\wedge_{n-1})=0$ and $\operatorname{Ext}^1_R(\mathcal{C}^\vee_{n-1},\mathcal{D})=0$.
	\end{proposition}

\subsection{Approximations}
Let $\mathcal{C}$ be a class in $R\operatorname{-Mod}$ and $M\in R\operatorname{-Mod}$.
A $R$-module homomorphism $f:C\to M$ is said to be a $\mathcal{C}$-\emph{precover} of $M$ if $C\in\mathcal{C}$ and if for every $R$-module homomorphism $f':C'\to M$ with $C'\in\mathcal{C}$, there exists a $R$-module homomorphism $h:C'\to C$ such that $f'=fh$.
Furthermore, an epimorphism  $f:C\to M$ with $C\in\mathcal{C}$ is called a \emph{special $\mathcal{C}$-precover} of $M$ if $\operatorname{Ker}f\in\mathcal{C}^{\perp_1}$.
Dually, we have the notions of $\mathcal{C}$-\emph{preenvelopes} and special $\mathcal{C}$-\emph{preenvelopes}.

Recall that a class of left $R$-modules is \emph{resolving} (resp. \emph{coresolving}) \cite{RJ2006} if it is closed under extensions and
kernels of epimorphisms (resp. cokernels of monomorphisms) and contains all projective (resp. injective) left
$R$-modules.

	\begin{lemma}\emph{\cite[Lemma 5.4]{Mao2020}}\label{lemma5.4}
Let $R$ be a ring and $0\to A_1\to A_2\to A_3\to 0$ be an exact sequence in $R$-$\operatorname{Mod}$.
		\begin{enumerate}[$(1)$]
			\item If $\mathcal{C}$ is a coresolving class in $R\operatorname{-Mod}$, $A_1$ and $A_3$ have special $\mathcal{C}$-preenvelopes, then $A_2$ has a special $\mathcal{C}$-preenvelope. In addition, there exists a commutative diagram with exact rows and columns as follows:
			$$\xymatrix{
				& 0\ar[d] & 0\ar[d] & 0\ar[d] & \\
				0 \ar[r] & A_1 \ar[r]\ar[d] & A_2 \ar[r]\ar[d] & A_3\ar[r]\ar[d] & 0\\
				0 \ar[r] & C_1 \ar[r]\ar[d] & C_2 \ar[r]\ar[d] & C_3\ar[r]\ar[d] & 0\\
				0 \ar[r] & B_1 \ar[r]\ar[d] & B_2 \ar[r]\ar[d] & B_3\ar[r]\ar[d] & 0\\
				& 0 & 0 & 0 &
			}$$
			where $C_i\in\mathcal{C}$ and $B_i\in{}^{\bot_1}\mathcal{C}$ for $i=1,2,3$.
			
			\item If $\mathcal{D}$ is a resolving class in $R\operatorname{-Mod}$, $A_1$ and $A_3$ have special $\mathcal{D}$-precovers, then $A_2$ has a special
			$\mathcal{D}$-precover. In addition, there exists  a commutative diagram with exact rows and columns as follows:
			$$\xymatrix{
				& 0\ar[d] & 0\ar[d] & 0\ar[d] & \\
				0 \ar[r] & E_1 \ar[r]\ar[d] & E_2 \ar[r]\ar[d] & E_3\ar[r]\ar[d] & 0\\
				0 \ar[r] & D_1 \ar[r]\ar[d] & D_2 \ar[r]\ar[d] & D_3\ar[r]\ar[d] & 0\\
				0 \ar[r] & A_1 \ar[r]\ar[d] & A_2 \ar[r]\ar[d] & A_3\ar[r]\ar[d] & 0\\
				& 0 & 0 & 0 &
			}$$
			where $D_i\in\mathcal{D}$ and $E_i\in\mathcal{D}^{\bot_1}$ for $i=1,2,3$.
		\end{enumerate}
	\end{lemma}

	\section{Main results}
\subsection{Special classes over formal triangular matrix rings}\label{Subsection 3.1}
\begin{definition}\rm\cite[P. 3]{Mao2020}\label{construction}
		Let $\mathcal{C}$ be a class in $A$-$\mathrm{Mod}$ and $\mathcal{D}$ be a class  in $B$-$\mathrm{Mod}$.
		Define $$\mathfrak{A}^\mathcal{C}_\mathcal{D}:=\{(\begin{smallmatrix}
			M_1\\M_2
		\end{smallmatrix})_{\varphi}\in \Lambda\operatorname{-Mod}\ |\  M_1\in\mathcal{C} \ \text{and} \ M_2\in\mathcal{D}\},$$
		$$\mathfrak{P}^\mathcal{C}_\mathcal{D}:=\{(\begin{smallmatrix}
			M_1\\M_2
		\end{smallmatrix})_{\varphi}\in \Lambda\operatorname{-Mod}\  |\  M_1\in\mathcal{C},\ \operatorname{Coker}\varphi\in\mathcal{D} \ \text{and} \ \varphi \ \text{is a monomorphism}\},$$
		$$\mathfrak{I}^\mathcal{C}_\mathcal{D}:=\{(\begin{smallmatrix}
			M_1\\M_2
		\end{smallmatrix})_{\varphi}\in \Lambda\operatorname{-Mod}\ |\  \operatorname{Ker}\widetilde{\varphi}\in\mathcal{C},\ M_2\in\mathcal{D} \ \text{and} \ \widetilde{\varphi} \ \text{is an epimorphism}\}.$$
	\end{definition}
		
	\begin{lemma}\emph{\cite[Lemma 2.4]{HMP2021}}\label{HMPLemma2.4}
Let $R$ be a ring and $\mathcal{C}$ be a class of $R$-$\mathrm{Mod}$.
Let $M$ be a left $R$-module and $N$ be a right $R$-module.
		\begin{enumerate}[$(1)$]
			
			\item If $\mathrm{Ext}_R^i(M,\mathcal{C})=0$ for  $1\leq i\leq n$,
			then $\mathrm{Ext}^1_R(M,\mathcal{C}^\wedge_{n-1})=0$.
			
			\item If $\mathrm{Tor}^R_i(N,\mathcal{C})=0$ for $1\leq i\leq n$,
			then $\mathrm{Tor}_1^R(N,\mathcal{C}_{n-1}^{\vee})=0$.
			
		\end{enumerate}
	\end{lemma}
	
	\begin{remark}
Lemma \ref{HMPLemma2.4}(1) can be obtained by \cite[Lemma 2.4]{HMP2021} directly. And through the analogous proof process of \cite[Lemma 2.4]{HMP2021}, ones can prove Lemma \ref{HMPLemma2.4}(2).
	\end{remark}

	\begin{lemma}\label{perp formula}
		Let $\mathcal{C},\mathcal{D}$ be classes in $A\operatorname{-Mod}$ and $\mathcal{E},\mathcal{F}$ be classes in $B$-$\mathrm{Mod}$.
		
\noindent$(1)$ If $\operatorname{Tor}^A_j(U,\mathcal{C})=0$ for $1\leq j\leq n+1$, then
		$\bigcap\limits^n_{i=1}(\mathfrak{P}^\mathcal{C}_\mathcal{E})^{\perp_i}=\mathfrak{A}_{\mathcal{E}^{\perp_{[1,n]}}}^{\mathcal{C}^{\perp_{[1,n]}}}$.

\noindent$(2)$ If $B^+\in\mathcal{F}$ and $\operatorname{Tor}_j^A(U,\bigcap\limits^n_{i=1}{}^{\perp_i}\mathcal{D})=0$ for $1\leq j\leq n+1$, then
		$\bigcap\limits^n_{i=1}{}^{\perp_i}(\mathfrak{A}^\mathcal{D}_\mathcal{F})=\mathfrak{P}^{{}^{\perp_{[1,n]}}\mathcal{D}}_{{}^{\perp_{[1,n]}}\mathcal{F}}$.
		
\noindent$(3)$ If $\operatorname{Ext}^j_B(U,\mathcal{F})=0$ for $1\leq j\leq n+1$, then $\bigcap\limits_{i=1}^n{}^{\perp_i}\mathfrak{I}^\mathcal{D}_\mathcal{F}=\mathfrak{A}^{{}^{\perp_{[1,n]}}\mathcal{D}}_{{}^{\perp_{[1,n]}}\mathcal{F}}$.

\noindent$(4)$ If $A\in\mathcal{C}$ and $\operatorname{Ext}^j_B(U,\bigcap\limits_{i=1}^n\mathcal{E}^{\perp_i})=0$ for $1\leq j\leq n+1$, then
$\bigcap\limits_{i=1}^n(\mathfrak{A}^\mathcal{C}_\mathcal{E})^{\perp_i}=\mathfrak{I}^{\mathcal{C}^{\perp_{[1,n]}}}_{\mathcal{E}^{\perp_{[1,n]}}}$.

	\end{lemma}
	
	\begin{proof}
		\noindent(1) ``$\subseteq$'' Let $\begin{pmatrix}
			X\\Y
		\end{pmatrix}_f\in\bigcap\limits^n_{i=1}(\mathfrak{P}^\mathcal{C}_\mathcal{E})^{\perp_i}$, $C\in\mathcal{C}$ and $E\in\mathcal{E}$.
		Since $\operatorname{Tor}^A_j(U,\mathcal{C})=0$ for $1\leq j\leq n+1$, then we have $\operatorname{Ext}_A^l(C,X)\cong\operatorname{Ext}^l_{\Lambda}(\begin{pmatrix}
			C\\U\otimes C
		\end{pmatrix},
		\begin{pmatrix}
			X\\Y
		\end{pmatrix})$
		for $1\leq l\leq n$, by Lemma \ref{extension formula}(3).
		Since $\begin{pmatrix}
			C\\U\otimes C
		\end{pmatrix}\in \mathfrak{P}^\mathcal{C}_\mathcal{E}$, then $\operatorname{Ext}^l_{\Lambda}(\begin{pmatrix}
			C\\U\otimes C
		\end{pmatrix},
		\begin{pmatrix}
			X\\Y
		\end{pmatrix})=0$
		for $1\leq l\leq n$.
		So $X\in\bigcap\limits^n_{i=1}\mathcal{C}^{\perp_i}$.
		Meanwhile, since $\begin{pmatrix}
			0\\E
		\end{pmatrix}\in\mathfrak{P}^\mathcal{C}_\mathcal{E}$, we have $\operatorname{Ext}^l_B(E,Y)\cong\operatorname{Ext}^l_{\Lambda}(\begin{pmatrix}
			0\\E
		\end{pmatrix},\begin{pmatrix}
			X\\Y
		\end{pmatrix})=0$ for $1\leq l\leq n$, by Lemma \ref{extension formula}(2).
		So $Y\in\bigcap\limits^n_{i=1}\mathcal{E}^{\perp_i}$.
		Hence $\begin{pmatrix}
			X\\Y
		\end{pmatrix}_f\in\mathfrak{A}_{\mathcal{E}^{\perp_{[1,n]}}}^{\mathcal{C}^{\perp_{[1,n]}}}$ and
		$\bigcap\limits^n_{i=1}(\mathfrak{P}^\mathcal{C}_\mathcal{E})^{\perp_i}\subseteq\mathfrak{A}_{\mathcal{E}^{\perp_{[1,n]}}}^{\mathcal{C}^{\perp_{[1,n]}}}$.
		
		``$\supseteq$'' Let $\begin{pmatrix}
			X\\Y
		\end{pmatrix}_f\in\mathfrak{A}_{\mathcal{E}^{\perp_{[1,n]}}}^{\mathcal{C}^{\perp_{[1,n]}}}$.
		For any $\begin{pmatrix}
			X'\\Y'
		\end{pmatrix}_{f'}\in\mathfrak{P}^{\mathcal{C}}_{\mathcal{E}}$,
		there exists a short exact sequence
		$$ 0\to\begin{pmatrix}
			X'\\U\otimes X'
		\end{pmatrix}_1\to\begin{pmatrix}
			X'\\Y'
		\end{pmatrix}_{f'}\to\begin{pmatrix}
			0\\ \operatorname{Coker}f'
		\end{pmatrix}_0\to 0
		$$
		in $\Lambda\operatorname{-Mod}$.
		Applying $\operatorname{Hom}_\Lambda(-,\begin{pmatrix}
			X\\Y
		\end{pmatrix})$ to this short exact sequence, we can obtain exact sequences
		$$
		\operatorname{Ext}^i_\Lambda(\begin{pmatrix}
			0\\ \operatorname{Coker}f'
		\end{pmatrix},\begin{pmatrix}
			X\\Y
		\end{pmatrix})\to
		\operatorname{Ext}^i_\Lambda(\begin{pmatrix}
			X'\\ Y'
		\end{pmatrix},\begin{pmatrix}
			X\\Y
		\end{pmatrix})\to
		\operatorname{Ext}^i_\Lambda(\begin{pmatrix}
			X'\\ U\otimes X'
		\end{pmatrix},\begin{pmatrix}
			X\\Y
		\end{pmatrix})
		$$
		for $1\leq i \leq n$.
		Since $\begin{pmatrix}
			X'\\Y'
		\end{pmatrix}_{f'}\in\mathfrak{P}^{\mathcal{C}}_{\mathcal{E}}$, we have $\operatorname{Coker}f'\in\mathcal{E}$.
		And, for $1\leq i\leq n$, we have $$\operatorname{Ext}^i_\Lambda(\begin{pmatrix}
			0\\ \operatorname{Coker}f'
		\end{pmatrix},\begin{pmatrix}
			X\\Y
		\end{pmatrix}\cong \operatorname{Ext}^i_B(\operatorname{Coker}f',Y)=0$$ by Lemma \ref{extension formula}(2).
		In addition, for $1\leq i\leq n$, we have $\operatorname{Ext}^i_\Lambda(\begin{pmatrix}
			X'\\ U\otimes X'
		\end{pmatrix},\begin{pmatrix}
			X\\Y
		\end{pmatrix})\cong\operatorname{Ext}^i_A(X',X)=0$ by Lemma \ref{extension formula}(3).
		Hence $\operatorname{Ext}^i_\Lambda(\begin{pmatrix}
			X'\\ Y'
		\end{pmatrix},\begin{pmatrix}
			X\\Y
		\end{pmatrix})=0$ for $1\leq i\leq n$ and $\begin{pmatrix}
			X\\Y
		\end{pmatrix}_f\in\bigcap\limits^n_{i=1}(\mathfrak{P}^\mathcal{C}_\mathcal{E})^{\perp_i}$.
		Therefore $\bigcap\limits^n_{i=1}(\mathfrak{P}^\mathcal{C}_\mathcal{E})^{\perp_i}\supseteq\mathfrak{A}_{\mathcal{E}^{\perp_{[1,n]}}}^{\mathcal{C}^{\perp_{[1,n]}}}$.
		
		\noindent(2)``$\subseteq$''
		Let $\begin{pmatrix}
			X\\Y
		\end{pmatrix}_f\in\bigcap\limits^n_{i=1}{}^{\perp_i}(\mathfrak{A}^\mathcal{D}_\mathcal{F})$.
		For any $D\in\mathcal{D}$, we have $\operatorname{Ext}^l_A(X,D)\cong\operatorname{Ext}^l_A(\begin{pmatrix}
			X\\Y
		\end{pmatrix},\begin{pmatrix}
			D\\0
		\end{pmatrix})=0$ for $1\leq l\leq n$ by Lemma \ref{extension formula}(1).
		So $X\in\bigcap\limits^n_{i=1}{}^{\perp_i}\mathcal{D}$.
		
		Applying $-\otimes_\Lambda\begin{pmatrix}
			X\\Y
		\end{pmatrix}_f$ to the short exact sequence
		$$0\to (U,0)\to (U,B)\to (0,B)\to 0$$ in $\operatorname{Mod-}\Lambda$,
		there exists an exact sequence in the first row of following diagram:
		$$\xymatrix@1{\operatorname{Tor}^\Lambda_1((0,B),{\begin{pmatrix}
					X\\Y
			\end{pmatrix}}) \ar[r] & (U,0)\otimes_\Lambda{\begin{pmatrix}
					X\\Y
			\end{pmatrix}} \ar[r]\ar[d]^{\cong} & (U,B)\otimes_\Lambda{\begin{pmatrix}
					X\\Y
			\end{pmatrix}}\ar[d]^{\cong} \\
			& U\otimes_A X \ar[r]^f & Y}$$
		By Lemma \ref{KT2017}, $(U,0)\otimes_\Lambda{\begin{pmatrix}
				X\\Y
		\end{pmatrix}}\cong U\otimes_A X$, $(U,B)\otimes_\Lambda\begin{pmatrix}
			X\\Y
		\end{pmatrix}\cong Y$ and the above square is commutative.
		By \cite[Corollary 10.63]{JJRotman},  $(\operatorname{Tor}^\Lambda_1((0,B),{\begin{pmatrix}
				X\\Y
		\end{pmatrix}}))^+\cong \operatorname{Ext}^1_\Lambda(\begin{pmatrix}
			X\\Y
		\end{pmatrix},\begin{pmatrix}
			0\\B^+
		\end{pmatrix})=0$.
		So $\operatorname{Tor}^\Lambda_1((0,B),{\begin{pmatrix}
				X\\Y
		\end{pmatrix}})=0$.
		And hence $f:U\otimes_A X\to Y$ is a momomorphism.
		
		It follows that there exists a short exact sequence:
		$$0\to \begin{pmatrix}
			X\\U\otimes_A X
		\end{pmatrix}_1\to\begin{pmatrix}
			X\\Y
		\end{pmatrix}_f\to\begin{pmatrix}
			0\\ \operatorname{Coker}f
		\end{pmatrix}_0\to 0$$ in $\Lambda\operatorname{-Mod}$.
		For any $F\in\mathcal{F}$, $\operatorname{Hom}_\Lambda(-,\begin{pmatrix}
			0\\F
		\end{pmatrix})$ induces the following exact sequence:
		$$\operatorname{Hom}_\Lambda(\begin{pmatrix}
			X\\U\otimes X
		\end{pmatrix},\begin{pmatrix}
			0\\F
		\end{pmatrix})\to\operatorname{Ext}^1_\Lambda(\begin{pmatrix}
			0\\ \operatorname{Coker}f
		\end{pmatrix},\begin{pmatrix}
			0\\F
		\end{pmatrix})\to\operatorname{Ext}^1_\Lambda(\begin{pmatrix}
			X\\Y
		\end{pmatrix},\begin{pmatrix}
			0\\F
		\end{pmatrix}).$$
		For any $\begin{pmatrix}
			0\\ \beta
		\end{pmatrix}:\begin{pmatrix}
			X\\ U\otimes X
		\end{pmatrix}\to\begin{pmatrix}
			0\\F
		\end{pmatrix}$, there exists a commutative diagram:
		$$\xymatrix{
			U\otimes_A X \ar[r]^0 \ar[d]^1 & U\otimes_A 0\ar[d]^0\\
			U\otimes_A X \ar[r]^\beta & F
		}$$
		
		It follows that $\beta=0$ and $\operatorname{Hom}_\Lambda(\begin{pmatrix}
			X\\U\otimes X
		\end{pmatrix},\begin{pmatrix}
			0\\F
		\end{pmatrix})=0$.
		And it is easy to see $\operatorname{Ext}^1_\Lambda(\begin{pmatrix}
			X\\Y
		\end{pmatrix},\begin{pmatrix}
			0\\F
		\end{pmatrix})=0$.
		So we have $\operatorname{Ext}^1_B(\operatorname{Coker}f,F)\cong\operatorname{Ext}^1_\Lambda(\begin{pmatrix}
			0\\ \operatorname{Coker}f
		\end{pmatrix},\begin{pmatrix}
			0\\F
		\end{pmatrix})=0$.
		We want to show $\operatorname{Coker}f\in\bigcap\limits^n_{i=1}{}^{\perp_i}\mathcal{F}$.
		Now, it remains to show $\operatorname{Ext}^l_B(\operatorname{Coker}f,F)=0$ for $2\leq l\leq n$. Consider the following exact sequences
		$$\operatorname{Ext}^{l-1}_\Lambda(\begin{pmatrix}
			X\\U\otimes_A X
		\end{pmatrix},\begin{pmatrix}
			0\\F
		\end{pmatrix})\to\operatorname{Ext}^{l}_\Lambda(\begin{pmatrix}
			0\\ \operatorname{Coker}f
		\end{pmatrix},\begin{pmatrix}
			0\\F
		\end{pmatrix})\to\operatorname{Ext}^{l}_\Lambda(\begin{pmatrix}
			X\\ Y
		\end{pmatrix},\begin{pmatrix}
			0\\F
		\end{pmatrix})$$
		for $2\leq l\leq n$.
		By Lemma \ref{extension formula}(3), we have $\operatorname{Ext}_\Lambda^{l-1}(\begin{pmatrix}
			X\\U\otimes_A X
		\end{pmatrix},\begin{pmatrix}
			0\\F
		\end{pmatrix})\cong\operatorname{Ext}_A^{l-1}(X,0)=0$ for $2\leq l\leq n$. In addition, since $\begin{pmatrix}
			X\\Y
		\end{pmatrix}\in\bigcap\limits^n_{i=1}{}^{\perp_i}(\mathfrak{A}^\mathcal{D}_\mathcal{F})$, we have $\operatorname{Ext}^{l}_\Lambda(\begin{pmatrix}
			X\\ Y
		\end{pmatrix},\begin{pmatrix}
			0\\F
		\end{pmatrix})=0$ for $2\leq l\leq n$.
		Hence $\operatorname{Ext}^l_B(\operatorname{Coker}f,F)\cong\operatorname{Ext}^{l}_\Lambda(\begin{pmatrix}
			0\\ \operatorname{Coker}f
		\end{pmatrix},\begin{pmatrix}
			0\\F
		\end{pmatrix})=0$ for $2\leq l\leq n$ by Lemma \ref{extension formula}(2).
		So $\operatorname{Coker}f\in\bigcap\limits^n_{i=1}{}^{\perp_i}\mathcal{F}$ and $\begin{pmatrix}
			X\\Y
		\end{pmatrix}_f\in\mathfrak{P}^{{}^{\perp_{[1,n]}}\mathcal{D}}_{{}^{\perp_{[1,n]}}\mathcal{F}}$.
		Therefore $\bigcap\limits^n_{i=1}{}^{\perp_i}(\mathfrak{A}^\mathcal{D}_\mathcal{F})\subseteq\mathfrak{P}^{{}^{\perp_{[1,n]}}\mathcal{D}}_{{}^{\perp_{[1,n]}}\mathcal{F}}$.
		
		\noindent``$\supseteq$'' Let $\begin{pmatrix}
			X\\Y
		\end{pmatrix}_f\in\mathfrak{P}^{{}^{\perp_{[1,n]}}\mathcal{D}}_{{}^{\perp_{[1,n]}}\mathcal{F}}$.
		For any $\begin{pmatrix}
			D\\F
		\end{pmatrix}_{f'}\in\mathfrak{A}^\mathcal{D}_\mathcal{F}$,
		the short exact sequence $$0\to \begin{pmatrix}
			X\\U\otimes_A X
		\end{pmatrix}_1\to\begin{pmatrix}
			X\\Y
		\end{pmatrix}_f\to\begin{pmatrix}
			0\\ \operatorname{Coker}f
		\end{pmatrix}_0\to 0$$
		induces the exact sequences
		$$\operatorname{Ext}^l_\Lambda(\begin{pmatrix}
			0\\ \operatorname{Coker}f
		\end{pmatrix},\begin{pmatrix}
			D\\F
		\end{pmatrix})\to\operatorname{Ext}^{l}_\Lambda(\begin{pmatrix}
			X\\ Y
		\end{pmatrix},\begin{pmatrix}
			D\\F
		\end{pmatrix})\to\operatorname{Ext}^{l}_\Lambda(\begin{pmatrix}
			X\\ U\otimes X
		\end{pmatrix},\begin{pmatrix}
			D\\F
		\end{pmatrix})$$
		for $1\leq l\leq n$.
		Since $\begin{pmatrix}
			X\\Y
		\end{pmatrix}_f\in\mathfrak{P}^{{}^{\perp_{[1,n]}}\mathcal{D}}_{{}^{\perp_{[1,n]}}\mathcal{F}}$, then $\operatorname{Coker}f\in\bigcap\limits_{i=1}^n{}^{\perp_i}\mathcal{F}$.
		By Lemma \ref{extension formula}(2), $$\operatorname{Ext}^l_\Lambda(\begin{pmatrix}
			0\\ \operatorname{Coker}f
		\end{pmatrix},\begin{pmatrix}
			D\\F
		\end{pmatrix})\cong \operatorname{Ext}^l_B(\operatorname{Coker}f,F)=0$$ for $1\leq l \leq n$.
		By Lemma \ref{extension formula}(3), $\operatorname{Ext}^{l}_\Lambda(\begin{pmatrix}
			X\\ U\otimes X
		\end{pmatrix},\begin{pmatrix}
			D\\F
		\end{pmatrix})\cong\operatorname{Ext}^l_A(X,D)=0$ for $1\leq l \leq n$.
		So $\operatorname{Ext}^{l}_\Lambda(\begin{pmatrix}
			X\\ Y
		\end{pmatrix},\begin{pmatrix}
			D\\F
		\end{pmatrix})=0$ for $1\leq l\leq n$ and $\begin{pmatrix}
			X\\Y
		\end{pmatrix}_f\in\bigcap\limits^n_{i=1}{}^{\perp_i}(\mathfrak{A}^\mathcal{D}_\mathcal{F})$.
		Therefore $\bigcap\limits^n_{i=1}{}^{\perp_i}(\mathfrak{A}^\mathcal{D}_\mathcal{F})\supseteq\mathfrak{P}^{{}^{\perp_{[1,n]}}\mathcal{D}}_{{}^{\perp_{[1,n]}}\mathcal{F}}$.

\noindent(3) The proof is similar to (1).

\noindent(4) The proof is similar to (2).
	\end{proof}

\begin{proposition}\label{prop5}
Let $R$ be a ring, $\mathcal{C}, \mathcal{F}$ be classes of left $R$-modules in $R\operatorname{-Mod}$.

\noindent$(1)$ Let $$0\to X\xrightarrow{d_0}C_0\xrightarrow{d_1}C_1\xrightarrow{
}\cdots\xrightarrow{} C_{n-1}\xrightarrow{d_n} C_n\to0$$ be an exact sequence in $R\operatorname{-Mod}$ with $C_i\in\mathcal{C}$ for $0\leq i\leq n$.
For any right $R$-module $M$,
if $\operatorname{Tor}^R_1(M,\mathcal{C}^\vee_n)=0$, then there exists an exact sequence $$ 0\to M\otimes_R X\xrightarrow{1\otimes_R d_0} M\otimes_R C_0\xrightarrow{1\otimes_R d_1} M\otimes_R C_1\xrightarrow{} \cdots\xrightarrow{}M\otimes_R C_{n-1}\xrightarrow{1\otimes_R d_n} M\otimes_R C_n\to 0.$$

\noindent$(2)$ Let $$0\to F_n\xrightarrow{f_n} F_{n-1}\to \cdots\to F_1\xrightarrow{f_1} F_0\xrightarrow{f_0} Y\to 0 $$ be an exact sequence in $R\operatorname{-Mod}$ with $F_i\in\mathcal{F}$ for $0\leq i\leq n$. For any left $R$-module $N$,
if $\operatorname{Ext}_R^1(N,\mathcal{F}^\wedge_n)=0$, then there exists an exact sequence
$$0\to\operatorname{Hom}_R(N,F_n)\xrightarrow{\overline{f_n}}\operatorname{Hom}_R(N,F_{n-1})\to\cdots\to\operatorname{Hom}_R(N,F_1)\xrightarrow{\overline{f_1}}\operatorname{Hom}_R(N,F_0)\xrightarrow{\overline{f_0}}\operatorname{Hom}_R(N,Y)\to 0,$$
where $\overline{f_i}=\operatorname{Hom}_R(N,f_i)$ for $0\leq i\leq n$.
\end{proposition}
\begin{proof}
$(1)$The short exact sequence $0\to X\xrightarrow{d_0}C_0\to \operatorname{Coker}d_0\to 0$ induces an exact sequence
$$\operatorname{Tor}^R_1(M,\operatorname{Coker}d_0)\to M\otimes_R X\to M\otimes_R C_0\to M\otimes_R \operatorname{Coker}d_0\to 0.$$
Clearly, $\operatorname{Coker}d_i\in\mathcal{C}^\vee_n$ for $0\leq i\leq n$.
So $\operatorname{Tor}^R_1(M,\operatorname{Coker}d_0)=0$ and there is a short exact sequence $$0\to M\otimes_R X\to M\otimes_R C_0\to M\otimes_R \operatorname{Coker}d_0\to 0.$$
For $0\leq i\leq n-2$, the short exact sequence $0\to \operatorname{Coker} d_i\to C_{i+1}\to \operatorname{Coker}d_{i+1}\to 0$ induces the exact sequence
$$\operatorname{Tor}^R_1(M,\operatorname{Coker}d_{i+1})\to M\otimes_R \operatorname{Coker}d_i\to M\otimes_R C_{i+1}\to M\otimes_R \operatorname{Coker}d_{i+1}\to 0$$
where $\operatorname{Tor}^R_1(M,\operatorname{Coker}d_{i+1})$ equals 0.
So there exist short exact sequences
$$0\to M\otimes_R \operatorname{Coker}d_i\to M\otimes_R C_{i+1}\to M\otimes_R \operatorname{Coker}d_{i+1}\to 0$$
for $0\leq i\leq n-2$.
Hence we have the exact sequence $$ 0\to M\otimes_R X\xrightarrow{1\otimes_R d_0} M\otimes_R C_0\xrightarrow{1\otimes_R d_1} M\otimes_R C_1\xrightarrow{} \cdots\xrightarrow{}M\otimes_R C_{n-1}\xrightarrow{1\otimes_R d_n} M\otimes_R C_n\to 0.$$

\noindent$(2)$ The proof is similar to $(1)$.
\end{proof}
	
\begin{lemma}\label{lemma 6}
		Let $\mathcal{C},\mathcal{D}$ be classes in $A\operatorname{-Mod}$ and $\mathcal{E},\mathcal{F}$ be classes in $B$-$\mathrm{Mod}$. The following statements hold.

\noindent$(1)$ $(\mathfrak{A}^\mathcal{C}_\mathcal{E})^\vee_n\subseteq\mathfrak{A}^{\mathcal{C}^\vee_n}_{\mathcal{E}^\vee_n}$. In addition, if $(\mathfrak{A}^\mathcal{C}_\mathcal{E})^\vee_n$ is closed under extensions, then $(\mathfrak{A}^\mathcal{C}_\mathcal{E})^\vee_n=\mathfrak{A}^{\mathcal{C}^\vee_n}_{\mathcal{E}^\vee_n}$ .

\noindent$(2)$ $(\mathfrak{A}^\mathcal{D}_\mathcal{F})^\wedge_n\subseteq\mathfrak{A}^{\mathcal{D}^\wedge_n}_{\mathcal{F}^\wedge_n}$. In addition, if $(\mathfrak{A}^\mathcal{D}_\mathcal{F})^\wedge_n$ is closed under extensions, then $(\mathfrak{A}^\mathcal{D}_\mathcal{F})^\wedge_n=\mathfrak{A}^{\mathcal{D}^\wedge_n}_{\mathcal{F}^\wedge_n}$ .

\noindent$(3)$ If $\operatorname{Tor}^A_1(U,\mathcal{C}^\vee_n)=0$, then $(\mathfrak{P}^\mathcal{C}_\mathcal{E})^\vee_n\subseteq\mathfrak{P}^{\mathcal{C}^\vee_n}_{\mathcal{E}^\vee_n}$.

\noindent$(4)$ If $\operatorname{Tor}^A_1(U,\mathcal{C}^\vee_n)=0$ and $(\mathfrak{P}^\mathcal{C}_\mathcal{E})^\vee_n$ is closed under extensions, then $\mathfrak{P}^{\mathcal{C}^\vee_n}_{\mathcal{E}^\vee_n}\subseteq(\mathfrak{P}^\mathcal{C}_\mathcal{E})^\vee_n$.

\noindent$(5)$ If $\operatorname{Ext}^1_B(U,\mathcal{F}^\wedge_n)=0$, then $(\mathfrak{I}^\mathcal{D}_\mathcal{F})^\wedge_n\subseteq\mathfrak{I}^{\mathcal{D}^\wedge_n}_{\mathcal{F}^\wedge_n}$.

\noindent$(6)$ If $\operatorname{Ext}^1_B(U,\mathcal{F}^\wedge_n)=0$ and $(\mathfrak{I}^\mathcal{D}_\mathcal{F})^\wedge_n$ is closed under extensions, then $\mathfrak{I}^{\mathcal{D}^\wedge_n}_{\mathcal{F}^\wedge_n}\subseteq(\mathfrak{I}^\mathcal{D}_\mathcal{F})^\wedge_n$.
\end{lemma}	

\begin{proof}
(1) For any $\begin{pmatrix}
	X\\Y
\end{pmatrix}_f\in(\mathfrak{A}^\mathcal{C}_\mathcal{E})^\vee_n$,
there exists an exact sequence
$$0\to\begin{pmatrix}
	X\\Y
\end{pmatrix}\to\begin{pmatrix}
	C_0\\E_0
\end{pmatrix}\to\begin{pmatrix}
	C_1\\E_1
\end{pmatrix}\to\cdots\to\begin{pmatrix}
	C_n\\E_n
\end{pmatrix}\to 0$$
in $\Lambda\operatorname{-Mod}$, where $\begin{pmatrix}
	C_i\\E_i
\end{pmatrix}\in\mathfrak{A}^\mathcal{C}_\mathcal{E}$ for $0\leq i\leq n$.
Then we have the exact sequences
$$ 0\to X\to C_0\to C_1\to\cdots\to C_n\to 0$$ in $A\operatorname{-Mod}$ and
$ 0\to Y\to E_0\to E_1\to\cdots\to  E_n\to 0$ in $B\operatorname{-Mod}$.
So $X\in\mathcal{C}^\vee_n$ and $Y\in E^\vee_n$.
Hence $\begin{pmatrix}
	X\\Y
\end{pmatrix}_f\in\mathfrak{A}^{\mathcal{C}^\vee_n}_{\mathcal{E}^\vee_n}$ and $(\mathfrak{A}^\mathcal{C}_\mathcal{E})^\vee_n\subseteq\mathfrak{A}^{\mathcal{C}^\vee_n}_{\mathcal{E}^\vee_n}$.

On the other hand, for any $\begin{pmatrix}
	X'\\Y'
\end{pmatrix}_{f'}\in\mathfrak{A}^{\mathcal{C}^\vee_n}_{\mathcal{E}^\vee_n}$,
there exists an exact sequence
$$0\to\begin{pmatrix}
	0\\Y'
\end{pmatrix}\to\begin{pmatrix}
X'\\Y'
\end{pmatrix}\to\begin{pmatrix}
X'\\0
\end{pmatrix}\to 0 $$ in $\Lambda\operatorname{-Mod}$.
Since $X'\in\mathcal{C}^\vee_n$, then there exists an exact sequence
$ 0\to X'\to C_0\to C_1\to\cdots\to C_n\to 0$ in $A\operatorname{-Mod}$,
where $C_i\in\mathcal{C}$ for $0\leq i\leq n$.
Further, we have the exact sequence
$$0\to\begin{pmatrix}
	X'\\0
\end{pmatrix}\to\begin{pmatrix}
C_0\\0
\end{pmatrix}\to\begin{pmatrix}
C_1\\0
\end{pmatrix}\to\cdots\to\begin{pmatrix}
C_n\\0
\end{pmatrix}\to 0.$$
So $\begin{pmatrix}
	X'\\0
\end{pmatrix}\in(\mathfrak{A}^\mathcal{C}_\mathcal{E})^\vee_n$.
Since $Y'\in\mathcal{E}^\vee_n$, there exists an exact sequence
$ 0\to Y'\to E_0\to E_1\to\cdots\to E_n\to 0$ in $B\operatorname{-Mod}$,
where $E_i\in\mathcal{E}$ for $0\leq i\leq n$.
Further, we have the exact sequence
$$0\to\begin{pmatrix}
	0\\Y'
\end{pmatrix}\to\begin{pmatrix}
	0\\E_0
\end{pmatrix}\to\begin{pmatrix}
	0\\E_1
\end{pmatrix}\to\cdots\to\begin{pmatrix}
	0\\E_n
\end{pmatrix}\to 0.$$
So $\begin{pmatrix}
	0\\Y'
\end{pmatrix}\in(\mathfrak{A}^\mathcal{C}_\mathcal{E})^\vee_n$.
Since $(\mathfrak{A}^\mathcal{C}_\mathcal{E})^\vee_n$ is closed under extensions, then
$\begin{pmatrix}
	X'\\Y'
\end{pmatrix}\in(\mathfrak{A}^\mathcal{C}_\mathcal{E})^\vee_n$.
Hence $\mathfrak{A}^{\mathcal{C}^\vee_n}_{\mathcal{E}^\vee_n}\subseteq(\mathfrak{A}^\mathcal{C}_\mathcal{E})^\vee_n$.

\noindent(2) The proof is similar to (1).
	
\noindent(3) For any $\begin{pmatrix}
		X\\Y
	\end{pmatrix}_f\in(\mathfrak{P}^\mathcal{C}_\mathcal{E})^\vee_n$,
there exists an exact sequence
$$ 0\to \begin{pmatrix}
	X\\Y
\end{pmatrix}_f\to\begin{pmatrix}
X_0\\Y_0
\end{pmatrix}_{f_0}\to\begin{pmatrix}
X_1\\Y_1
\end{pmatrix}_{f_1}\to\cdots\to\begin{pmatrix}
X_n\\Y_n
\end{pmatrix}_{f_n}\to0\ \ \ \ \ \ (\triangle)$$ in $\Lambda\operatorname{-Mod}$,
where $\begin{pmatrix}
	X_i\\Y_i
\end{pmatrix}_{f_i}\in\mathfrak{P}^\mathcal{C}_\mathcal{E}$ for $0\leq i\leq n$.
Then we have the exact sequence $$0\to X\to X_0\to X_1\to \cdots\to X_n\to 0$$
in $A\operatorname{-Mod}$, where $X_i\in\mathcal{C}$ for $0\leq i\leq n$.
So $X\in\mathcal{C}^\vee_n$.
Since $\begin{pmatrix}
	X_i\\Y_i
\end{pmatrix}_{f_i}\in\mathfrak{P}^\mathcal{C}_\mathcal{E}$, then
 $f_i$ is a monomorphism and $\operatorname{Coker}f_i\in\mathcal{E}$ for $0\leq i\leq n$.
Consider the following diagram:
$$ \xymatrix{0\ar[r] & U\otimes_A X  \ar[r] \ar@{^{(}->}[d]_f & U\otimes_A X_0 \ar@{^{(}->}[d]_{f_0} \ar[r] & U\otimes_A X_1 \ar@{^{(}->}[d]_{f_1} \ar[r] & \cdots \ar[r] & U\otimes_A X_{n-1}\ar@{^{(}->}[d]_{f_{n-1}}\ar[r]& U\otimes_A X_n\ar@{^{(}->}[d]_{f_n}\ar[r] & 0	\\
0\ar[r] & Y\ar@{->>}[d] \ar[r] & Y_0\ar@{->>}[d] \ar[r] & Y_1\ar@{->>}[d] \ar[r] & \cdots \ar[r] & Y_{n-1}\ar@{->>}[d]\ar[r]& Y_n\ar@{->>}[d]\ar[r] & 0\\
0\ar[r] & \operatorname{Coker}f \ar[r] & \operatorname{Coker}f_0 \ar[r] & \operatorname{Coker}f_1 \ar[r] & \cdots \ar[r] & \operatorname{Coker}f_{n-1}\ar[r]& \operatorname{Coker}f_n\ar[r] & 0}$$
Due to the exact sequence $(\triangle)$, the second row is exact and the upper squares are commutative.
By Proposition \ref{prop5}(1), the first row is exact.
In addition, since $f_0$ is a monomorphism, then $f$ is also a monomorphism.
Then the exact sequences in the first and second rows induces that the third row is also exact.
It follows that  $\operatorname{Coker}f\in\mathcal{E}^\vee_n$.
Hence $\begin{pmatrix}
	X\\Y
\end{pmatrix}_f\in\mathfrak{P}^{\mathcal{C}^\vee_n}_{\mathcal{E}^\vee_n}$
and $(\mathfrak{P}^\mathcal{C}_\mathcal{E})^\vee_n\subseteq\mathfrak{P}^{\mathcal{C}^\vee_n}_{\mathcal{E}^\vee_n}$.

\noindent(4) For any $\begin{pmatrix}
	X\\Y
\end{pmatrix}_f\in\mathfrak{P}^{\mathcal{C}^\vee_n}_{\mathcal{E}^\vee_n}$, we can see that $X\in\mathcal{C}^\vee_n$, $\operatorname{Coker}f\in\mathcal{E}^\vee_n$ and $f$ is a monomorphism.
Consider the following exact sequence:
$$0\to\begin{pmatrix}
X\\ U\otimes_A X
\end{pmatrix}_1\to \begin{pmatrix}
X\\Y
\end{pmatrix}_f\to\begin{pmatrix}
0\\ \operatorname{Coker}f
\end{pmatrix}_0\to 0 $$ in $\Lambda\operatorname{-Mod}$.
Since $\operatorname{Coker}f\in\mathcal{E}^\vee_n$, then there exists an exact sequence
$$0\to \operatorname{Coker}f\to E_0\to E_1\to\cdots\to E_n\to 0$$ with $E_i\in\mathcal{E}$ for $0\leq i\leq n$, which induces the exact sequence
$$0\to\begin{pmatrix}
0\\ \operatorname{Coker}f
\end{pmatrix}\to\begin{pmatrix}
0\\E_0
\end{pmatrix}\to\begin{pmatrix}
0\\E_1
\end{pmatrix}\to\cdots\to\begin{pmatrix}
0\\E_n
\end{pmatrix}\to 0 $$ in $\Lambda\operatorname{-Mod}$.
Clearly, $\begin{pmatrix}
	0\\E_i
\end{pmatrix}\in\mathfrak{P}^\mathcal{C}_\mathcal{E}$ for $0\leq i\leq n$.
So $\begin{pmatrix}
	0\\ \operatorname{Coker}f
\end{pmatrix}\in(\mathfrak{P}^\mathcal{C}_\mathcal{E})^\vee_n$.

Since $X\in\mathcal{C}^\vee_n$, there exists an exact sequence
$$0\to X\to C_0\to C_1\to \cdots \to C_n\to 0$$ in $A\operatorname{-Mod}$, where $C_i\in\mathcal{C}$.
By Proposition \ref{prop5}(1), there exists an exact sequence
$$0\to U\otimes_A X\to U\otimes_A C_0\to U\otimes_A C_1\to \cdots\to U\otimes_A C_n\to 0  $$ in $B\operatorname{-Mod}$.
Further, there exists the following exact sequence
$$0\to\begin{pmatrix}
	X\\U\otimes_A X
\end{pmatrix}\to\begin{pmatrix}
C_0\\U\otimes_A C_0
\end{pmatrix}\to\begin{pmatrix}
C_1\\U\otimes_A C_1
\end{pmatrix}\to \cdots\to\begin{pmatrix}
C_n\\U\otimes_A C_n
\end{pmatrix}\to 0$$ in $\Lambda\operatorname{-Mod}$.
Clearly, $\begin{pmatrix}
	C_i\\U\otimes_A C_i
\end{pmatrix}\in\mathfrak{P}^\mathcal{C}_\mathcal{E}$ for $0\leq i\leq n$.
So $\begin{pmatrix}
	X\\ U\otimes_A X
\end{pmatrix}\in(\mathfrak{P}^\mathcal{C}_\mathcal{E})^\vee_n$.
Since $(\mathfrak{P}^\mathcal{C}_\mathcal{E})^\vee_n$ is closed under extensions, then $\begin{pmatrix}
	X\\Y
\end{pmatrix}_f\in(\mathfrak{P}^\mathcal{C}_\mathcal{E})^\vee_n$.
Hence $\mathfrak{P}^{\mathcal{C}^\vee_n}_{\mathcal{E}^\vee_n}\subseteq
(\mathfrak{P}^\mathcal{C}_\mathcal{E})^\vee_n$

\noindent(5) The proof is similar to (3).

\noindent(6) The proof is similar to (4).
\end{proof}

\subsection{Left (right) $n$-cotorsion pairs over formal triangular matrix rings}
\label{Subsection 3.2}
	\begin{theorem}\label{main result1}
\noindent$(1)$Let $(\mathcal{C},\mathcal{D})$ be a right $n$-cotorsion pair in $A\operatorname{-Mod}$ and $(\mathcal{E},\mathcal{F})$ be a right $n$-cotorsion pair in $B\operatorname{-Mod}$.
		If $\operatorname{Tor}^A_j(U,\mathcal{C})=0$ for $1\leq j\leq n+1$, $\mathfrak{A}_\mathcal{F} ^\mathcal{D}$ is a coresolving class and $(\mathfrak{P}^\mathcal{C}_\mathcal{E})^\vee_{n-1}$ is
		closed under extensions, then $(\mathfrak{P}^\mathcal{C}_\mathcal{E},\mathfrak{A}^\mathcal{D}_\mathcal{F})$ is a right $n$-cotorsion pair in $\Lambda\operatorname{-Mod}$.
		
\noindent$(2)$Let $(\mathcal{C},\mathcal{D})$ be a left $n$-cotorsion pair in $A\operatorname{-Mod}$ and $(\mathcal{E},\mathcal{F})$ be a left $n$-cotorsion pair in $B\operatorname{-Mod}$.
If $\operatorname{Ext}_B^j(U,\mathcal{F})=0$ for $1\leq j\leq n+1$,
$\mathfrak{A}^\mathcal{C}_\mathcal{E}$ is a resolving class and $(\mathfrak{I}^\mathcal{D}_\mathcal{F})^\wedge_{n-1}$ is closed under extensions, then $(\mathfrak{A}^\mathcal{C}_\mathcal{E},\mathfrak{I}^\mathcal{D}_\mathcal{F})$ is a left $n$-cotorsion pair in in $\Lambda\operatorname{-Mod}$.

	\end{theorem}
	
	\begin{proof}
\noindent$(1)$ Since $(\mathcal{C},\mathcal{D})$ is a right $n$-cotorsion pair in  $A\operatorname{-Mod}$ and $(\mathcal{E},\mathcal{F})$ is a right $n$-cotorsion pair in $B\operatorname{-Mod}$, then we have $\bigcap\limits^n_{i=1}\mathcal{C}^{\perp_i}=\mathcal{D}$ and $\bigcap\limits^n_{i=1}\mathcal{E}^{\perp_i}=\mathcal{F}$ by Theorem \ref{n-cotorsion pair2}.
		Then $\bigcap\limits^n_{i=1}(\mathfrak{P}^\mathcal{C}_\mathcal{E})^{\perp_i}=\mathfrak{A}^{\mathcal{C}^{\perp_{[1,n]}}}_{\mathcal{E}^{\perp_{[1,n]}}}=\mathfrak{A}^\mathcal{D}_\mathcal{F}$ by Lemma \ref{perp formula}(1).
		And then we have $\operatorname{Ext}^1_\Lambda((\mathfrak{P}^\mathcal{C}_\mathcal{E})^\vee_{n-1},\mathfrak{A}^\mathcal{D}_\mathcal{F})=0$ and $(\mathfrak{P}^\mathcal{C}_\mathcal{E})^\vee_{n-1}\subseteq{}^{\perp_1}(\mathfrak{A}^\mathcal{D}_\mathcal{F})$ by Proposition \ref{prop2.3}.
		
		For any $\begin{pmatrix}
			M_1\\M_2
		\end{pmatrix}_{\varphi^M}\in\Lambda\operatorname{-Mod}$, there exists a short exact sequence
		$$0\to\begin{pmatrix}
			0\\M_2
		\end{pmatrix}\to\begin{pmatrix}
			M_1\\M_2
		\end{pmatrix}\to\begin{pmatrix}
			M_1\\0
		\end{pmatrix}\to 0$$
		in $\Lambda\operatorname{-Mod}$.
		Since $(\mathcal{C},\mathcal{D})$ is a right $n$-cotorsion pair in $A\operatorname{-Mod}$, there exists a short exact sequence
		$$0\to M_1\xrightarrow{l} D_1\xrightarrow{\pi} C_1\to 0$$ in $A\operatorname{-Mod}$, where $D_1\in\mathcal{D}$ and $C_1\in\mathcal{C}_{n-1}^\vee$.
Since $(\mathcal{E},\mathcal{F})$ is a right $n$-cotorsion pair in $B\operatorname{-Mod}$, there exists a short exact sequence $$0\to U\otimes_A C_1\xrightarrow{\lambda} F_1\to E_1\to 0$$ in $B\operatorname{-Mod}$, where $F_1\in\mathcal{F}$ and $E_1\in\mathcal{E}^\vee_{n-1}$.
Let $N:=\begin{pmatrix}
D_1\\F_1
\end{pmatrix}_{\varphi^N}$ with $\varphi^N=\lambda\circ(1\otimes \pi)$.
Then there exists a short exact sequence
$$0\to \begin{pmatrix}
	M_1\\0
\end{pmatrix}\xrightarrow{(\begin{smallmatrix}
	l\\0
	\end{smallmatrix})}
\begin{pmatrix}
D_1\\F_1
\end{pmatrix}_{\varphi^N}\xrightarrow{(\begin{smallmatrix}
	\pi\\1
	\end{smallmatrix})}
\begin{pmatrix}
	C_1\\F_1
\end{pmatrix}_\lambda\to 0 $$
where $N=\begin{pmatrix}
	D_1\\F_1
\end{pmatrix}_{\varphi^N}\in\mathfrak{A}^\mathcal{D}_\mathcal{F}$.
In addition, $\begin{pmatrix}
	C_1\\F_1
\end{pmatrix}_\lambda\in\mathfrak{P}^{\mathcal{C}^\vee_{n-1}}_{\mathcal{E}^\vee_{n-1}}=(\mathfrak{P}^\mathcal{C}_\mathcal{E})^\vee_{n-1}\subseteq{}^{\perp_1}(\mathfrak{A}^\mathcal{D}_\mathcal{F})$ by Lemma \ref{lemma 6} and Lemma \ref{HMPLemma2.4}(2).
So $\begin{pmatrix}
M_1\\0
\end{pmatrix}$
has a special $\mathfrak{A}^\mathcal{D}_\mathcal{F}$-preenvelope.
For $M_2\in B\operatorname{-Mod}$, there exists an exact sequence
$$0\to M_2\to F_2\to E_2\to 0$$ where $F_2\in\mathcal{F}$ and $E_2\in\mathcal{E}^\vee_{n-1}$.
And then there exists an exact sequence
$$0\to\begin{pmatrix}
	0\\M_2
\end{pmatrix}\to\begin{pmatrix}
0\\F_2
\end{pmatrix}\to\begin{pmatrix}
0\\E_2
\end{pmatrix}\to 0 $$
with $\begin{pmatrix}
	0\\F_2
\end{pmatrix}\in\mathfrak{A}^\mathcal{D}_\mathcal{F}$.
In addition,
$\begin{pmatrix}
0\\E_2
\end{pmatrix}\in\mathfrak{P}^{\mathcal{C}^\vee_{n-1}}_{\mathcal{E}^\vee_{n-1}}=(\mathfrak{P}^\mathcal{C}_\mathcal{E})^\vee_{n-1}\subseteq{}^{\perp_1}(\mathfrak{A}^\mathcal{D}_\mathcal{F})$ by Lemma \ref{lemma 6} and Lemma \ref{HMPLemma2.4}(2).
So $\begin{pmatrix}
	0\\M_2
\end{pmatrix}$
also has a special $\mathfrak{A}^\mathcal{D}_\mathcal{F}$-preenvelope.
By Lemma \ref{lemma5.4}, we have the following commutative diagram with exact rows and columns:
$$\xymatrix{ & 0\ar[d] & 0\ar[d] & 0\ar[d] & \\
0\ar[r] & {\begin{pmatrix}
	0\\M_2
	\end{pmatrix}} \ar[r]\ar[d] & {\begin{pmatrix}
	M_1\\M_2
\end{pmatrix}} \ar[r]\ar[d] & {\begin{pmatrix}
M_1\\0
\end{pmatrix}}\ar[r]\ar[d] & 0\\
0\ar[r] & {\begin{pmatrix}
		0\\F_2
\end{pmatrix}} \ar[r]\ar[d] & {\begin{pmatrix}
		D_1\\F'
\end{pmatrix}} \ar[r]\ar[d] & {\begin{pmatrix}
		D_1\\F_1
\end{pmatrix}}\ar[r]\ar[d] & 0\\
0\ar[r] & {\begin{pmatrix}
		0\\E_2
\end{pmatrix}} \ar[r]\ar[d] & {\begin{pmatrix}
		C_1\\Y
\end{pmatrix}} \ar[r]\ar[d] & {\begin{pmatrix}
		C_1\\F_1
\end{pmatrix}}\ar[r]\ar[d] & 0\\
 & 0 & 0 & 0 &} $$

\noindent Since $\mathcal{F}$ is closed under extensions, then $F'\in\mathcal{F}$ and $\begin{pmatrix}
	D_1\\F'
\end{pmatrix}\in\mathfrak{A}^\mathcal{D}_\mathcal{F}$.
Since $(\mathfrak{P}^\mathcal{C}_\mathcal{E})^\vee_{n-1}$ is closed under extensions, then $\begin{pmatrix}
	C_1\\Y
\end{pmatrix}\in(\mathfrak{P}^\mathcal{C}_\mathcal{E})^\vee_{n-1}$.
Hence $(\mathfrak{P}^\mathcal{C}_\mathcal{E},\mathfrak{A}^\mathcal{D}_\mathcal{F})$ is a right $n$-cotorsion pair in $\Lambda\operatorname{-Mod}$.\\
$(2)$ The proof is similar to (1).
\end{proof}

\begin{theorem}\label{main result3}
Let $\mathcal{C},\mathcal{D}$ be classes in $A\operatorname{-Mod}$ and $\mathcal{E},\mathcal{F}$ be classes in $B\operatorname{-Mod}$.\\
$(1)$ Assume $(\mathfrak{P}^\mathcal{C}_\mathcal{E},\mathfrak{A}^\mathcal{D}_\mathcal{F})$ is a right $n$-cotorsion pair in $\Lambda\operatorname{-Mod}$.

$(a)$ If $\operatorname{Tor}^A_j(U,\mathcal{C})=0$ for $1\leq j\leq n+1$, then $(\mathcal{C},\mathcal{D})$ is a right $n$-cotorsion pair in $A\operatorname{-Mod}$.

$(b)$ If $\operatorname{Tor}^A_j(U,\mathcal{C})=0$ for $1\leq j\leq n+1$, $\mathcal{E}^\vee_{n-1}$ is closed under extensions and $U\otimes_A \mathcal{D}\subseteq \mathcal{E}^\vee_{n-1}$, then $(\mathcal{E},\mathcal{F})$ is a right $n$-cotorsion pair in $B\operatorname{-Mod}$.\\
$(2)$ Assume $(\mathfrak{A}^\mathcal{C}_\mathcal{E},\mathfrak{I}^\mathcal{D}_\mathcal{F})$ is a left $n$-cotorsion pair in $\Lambda\operatorname{-Mod}$.

$(a)$ If $\operatorname{Ext}^j_B(U,\mathcal{F})=0$ for $1\leq j\leq n+1$, then $(\mathcal{E},\mathcal{F})$ is a left $n$-cotorsion pair in $B\operatorname{-Mod}$.

$(b)$ If $\operatorname{Ext}^j_B(U,\mathcal{F})=0$ for $1\leq j\leq n+1$, $\mathcal{D}^\wedge_{n-1}$ is closed under extensions and $\operatorname{Hom}_B(U,\mathcal{E})\subseteq\mathcal{D}^\wedge_{n-1}$, then $(\mathcal{C},\mathcal{D})$ is a left $n$-cotorsion pair in $A\operatorname{-Mod}$.
\end{theorem}
\begin{proof}
(1) (a) For any $C\in\mathcal{C}, D\in\mathcal{D}$ and $1\leq j\leq n$, we have $\operatorname{Ext}_A^j(C,D)\cong\operatorname{Ext}^j_\Lambda(\begin{pmatrix}
	C\\U\otimes_A C
\end{pmatrix},\begin{pmatrix}
D\\0
\end{pmatrix})=0$ by Lemma \ref{extension formula}(3). So $\mathcal{D}\subseteq\bigcap\limits_{i=1}^n\mathcal{C}^{\perp_i}$.
For any $N\in\bigcap\limits_{i=1}^n\mathcal{C}^{\perp_i}$, we have $\begin{pmatrix}
	N\\0
\end{pmatrix}\in\mathfrak{A}^{\mathcal{C}^{\perp_{[1,n]}}}_{\mathcal{E}^{\perp_{[1,n]}}}=\bigcap\limits_{i=1}^n(\mathfrak{P}^\mathcal{C}_\mathcal{E})^{\perp_i}=\mathfrak{A}^\mathcal{D}_\mathcal{F}$ by Lemma \ref{perp formula}(1) and Theorem \ref{n-cotorsion pair2}(2).
So $N\in\mathcal{D}$ and $\bigcap\limits_{i=1}^n\mathcal{C}^{\perp_i}\subseteq\mathcal{D}$.
Hence $\bigcap\limits_{i=1}^n\mathcal{C}^{\perp_i}=\mathcal{D}$.

For any $M\in A\operatorname{-Mod}$, since $(\mathfrak{P}^\mathcal{C}_\mathcal{E},\mathfrak{A}^\mathcal{D}_\mathcal{F})$ is a right $n$-cotorsion pair in $\Lambda\operatorname{-Mod}$, then there exists an exact sequence
$0\to\begin{pmatrix}
	M\\0
\end{pmatrix}\to\begin{pmatrix}
X\\Y
\end{pmatrix}_f\to\begin{pmatrix}
X'\\Y'
\end{pmatrix}_{f'}\to 0$
in $\Lambda\operatorname{-Mod}$, where $\begin{pmatrix}
	X\\Y
\end{pmatrix}\in\mathfrak{A}^\mathcal{D}_\mathcal{F}$ and $\begin{pmatrix}
X'\\Y'
\end{pmatrix}\in(\mathfrak{P}^\mathcal{C}_\mathcal{E})^\vee_{n-1}$.
In addition, $\begin{pmatrix}
	X'\\Y'
\end{pmatrix}\in(\mathfrak{P}^\mathcal{C}_\mathcal{E})^\vee_{n-1}\subseteq\mathfrak{P}^{\mathcal{C}^\vee_{n-1}}_{\mathcal{E}^\vee_{n-1}}$ by Lemma \ref{lemma 6}(3) and Lemma \ref{HMPLemma2.4}(2).
Then we have the short exact sequence $0\to M\to X\to X'\to 0$ in $A\operatorname{-Mod}$, where  $X\in\mathcal{D}$ and $X'\in\mathcal{C}^\vee_{n-1}$.
Hence $(\mathcal{C},\mathcal{D})$ is a right $n$-cotorsion pair in $A\operatorname{-Mod}$.

(b) For any $E\in\mathcal{E}$, $F\in\mathcal{F}$ and $1\leq i\leq n$, we have $\operatorname{Ext}_B^i(E,F)\cong\operatorname{Ext}^i_\Lambda(\begin{pmatrix}
	0\\E
\end{pmatrix},\begin{pmatrix}
0\\F
\end{pmatrix})=0$ by Lemma \ref{extension formula}(2).
So $\mathcal{F}\subseteq\bigcap\limits_{i=1}^n\mathcal{E}^{\perp_i}$.
For any $X\in\bigcap\limits_{i=1}^n\mathcal{E}^{\perp_i}$, we have $\begin{pmatrix}
	0\\X
\end{pmatrix}\in\mathfrak{A}^{\mathcal{C}^{\perp_[1,n]}}_{\mathcal{E}^{\perp_[1,n]}}=\bigcap\limits_{i=1}^n(\mathfrak{P}^\mathcal{C}_\mathcal{E})^{\perp_i}=\mathfrak{A}^\mathcal{D}_\mathcal{F}$ by Lemma \ref{perp formula}(1). So $X\in\mathcal{F}$ and $\bigcap\limits_{i=1}^n\mathcal{E}^{\perp_i}\subseteq\mathcal{F}$.
Hence $\bigcap\limits_{i=1}^n\mathcal{E}^{\perp_i}=\mathcal{F}$.

For any $M\in B\operatorname{-Mod}$, there exists an exact sequence
$$0\to\begin{pmatrix}
	0\\M
\end{pmatrix}\to\begin{pmatrix}
X_1\\Y_1
\end{pmatrix}_{f_1}\to\begin{pmatrix}
X_2\\Y_2
\end{pmatrix}_{f_2}\to 0,$$
where $\begin{pmatrix}
	X_1\\Y_1
\end{pmatrix}\in\mathfrak{A}^\mathcal{D}_\mathcal{F}$ and $\begin{pmatrix}
X_2\\Y_2
\end{pmatrix}\in(\mathfrak{P}^\mathcal{C}_\mathcal{E})^\vee_{n-1}$.
In addition, $\begin{pmatrix}
	X_2\\Y_2
\end{pmatrix}\in(\mathfrak{P}^\mathcal{C}_\mathcal{E})^\vee_{n-1}\subseteq\mathfrak{P}^{\mathcal{C}^\vee_{n-1}}_{\mathcal{E}^\vee_{n-1}}$ by Lemma \ref{lemma 6}(3) and Lemma \ref{HMPLemma2.4}(2).
Then we have the exact sequence $$0\to U\otimes_A X_2\xrightarrow{f_2}Y_2\to \operatorname{Coker}f_2\to 0$$
in $B\operatorname{-Mod}$, where $\operatorname{Coker}f_2\in\mathcal{E}^\vee_{n-1}$.
By the condition $U\otimes_A \mathcal{D}\subseteq \mathcal{E}^\vee_{n-1}$, then $U\otimes_A X_2\cong U\otimes_A X_1\in \mathcal{E}^\vee_{n-1}$. Since $ \mathcal{E}^\vee_{n-1}$ is closed under extension, then $Y_2\in\mathcal{E}^\vee_{n-1}$. So there exists an exact sequence
$$0\to M\to Y_1\to Y_2\to 0$$ in $B\operatorname{-Mod}$ with $Y_1\in \mathcal{F}$ and $Y_2\in\mathcal{E}^\vee_{n-1}$.
Hence $(\mathcal{E},\mathcal{F})$ is a right $n$-cotorsion pair in $B\operatorname{-Mod}$.\\
(2) The proof is similar to (1).
\end{proof}

\subsection{Hereditary left (right) $n$-cotorsion pairs}
\label{Subsection 3.3}

\begin{definition}\rm\cite[Definition 4.8]{HMP2021}
Let $R$ be a ring.
A right (left) $n$-cotorsion pair $(\mathcal{C},\mathcal{D})$ in $R\operatorname{-Mod}$ is \emph{hereditary} if $\operatorname{Ext}^{n+1}_R(\mathcal{C},\mathcal{D})=0$.
\end{definition}
We give a completion of \cite[Proposition 4.9]{HMP2021} in the following lemma.
\begin{lemma}\label{hereditary equiv}
Let $R$ be a ring.\\
$(1)$Let $(\mathcal{C},\mathcal{D})$ be a right $n$-cotorsion pair in $R\operatorname{-Mod}$. Then the following conditions are equivalent:

\emph{(a)} $\operatorname{Ext}^i_R(\mathcal{C},\mathcal{D})=0$ for $i\geq 1$.

\emph{(b)} $(\mathcal{C},\mathcal{D})$ is hereditary, i.e., $\operatorname{Ext}_R^{n+1}(\mathcal{C},\mathcal{D})=0$.

\emph{(c)} $\mathcal{D}$ is a coresolving class in $R\operatorname{-Mod}$.

\emph{(d)} $\mathcal{D}$ is closed under cokernels of monomorphisms.\\
$(2)$Let $(\mathcal{C},\mathcal{D})$ be a left $n$-cotorsion pair in $R\operatorname{-Mod}$. Then the following conditions are equivalent:

\emph{(a)} $\operatorname{Ext}^i_R(\mathcal{C},\mathcal{D})=0$ for $i\geq 1$.

\emph{(b)} $(\mathcal{C},\mathcal{D})$ is hereditary, i.e., $\operatorname{Ext}_R^{n+1}(\mathcal{C},\mathcal{D})=0$.

\emph{(c)} $\mathcal{C}$ is a resolving class in $R\operatorname{-Mod}$.

\emph{(d)} $\mathcal{C}$ is closed under kernels of epimorphisms.
\end{lemma}
\begin{proof}
As the proof of (1) is similar to (2), we only prove (2).

It is easy to see that (a) implies (b). And due to \cite[Proposition 4.9]{HMP2021}, (b) implies (c). By the definition of the resolving class, (c) implies (d).
It remains to show (d) implies (a).

Let $C\in\mathcal{C}$ and $D\in\mathcal{D}$.
Since $(\mathcal{C},\mathcal{D})$ is a left $n$-cotorsion pair, then $\operatorname{Ext}_R^i(C,D)=0$ for $1\leq i\leq n$.
For any $j\geq 1$, we have
$\operatorname{Ext}^{n+j}_R(C,D)\cong \operatorname{Ext}^n_R(\Omega^j(C),D)$ by dimension shifting, where $\Omega^j(C)$ is the $j$-th syzygy of $C$.
Since $\mathcal{C}$ is closed under kernels of epimorphisms and all projective $R$-modules belong to $\mathcal{C}$, then $\Omega^j(C)\in\mathcal{C}$. So $\operatorname{Ext}^n_R(\Omega^j(C),D)=0$ for $j\geq 1$.
Hence $\operatorname{Ext}^i_R(\mathcal{C},\mathcal{D})=0$ for $i\geq 1$.
\end{proof}
\begin{lemma}\label{hereditary eq2}
$(1)$ Let $(\mathcal{C},\mathcal{D})$ be a right $n$-cotorsion pair in $A\operatorname{-Mod}$ and $(\mathcal{E},\mathcal{F})$ be a right $n$-cotorsion pair in $B\operatorname{-Mod}$.
Assume $\operatorname{Tor}^A_i(U,\mathcal{C})=0$ for $1\leq i\leq n$.
Then $\mathfrak{A}^\mathcal{D}_\mathcal{F}$ is a coresolving class in $\Lambda\operatorname{-Mod}$ if and only if
$(\mathcal{C},\mathcal{D})$ and $(\mathcal{E},\mathcal{F})$ are both hereditary.\\
$(2)$ Let $(\mathcal{C},\mathcal{D})$ be a left $n$-cotorsion pair in $A\operatorname{-Mod}$ and $(\mathcal{E},\mathcal{F})$ be a left $n$-cotorsion pair in $B\operatorname{-Mod}$.
Assume $\operatorname{Ext}^i_B(U,\mathcal{F})=0$ for $1\leq i\leq n$.
Then $\mathfrak{A}^\mathcal{C}_\mathcal{E}$ is a resolving class in $\Lambda\operatorname{-Mod}$ if and only if
$(\mathcal{C},\mathcal{D})$ and $(\mathcal{E},\mathcal{F})$ are both hereditary.
\end{lemma}
\begin{proof}
(1) ``$\Rightarrow$'' Since $\mathfrak{A}^\mathcal{D}_\mathcal{F}$ is a coresolving class in $\Lambda\operatorname{-Mod}$, then $\mathfrak{A}^\mathcal{D}_\mathcal{F}$ is  closed under cokernels of monomorphisms. This implies that $\mathcal{D}$ and $\mathcal{F}$ are both closed under cokernels of monomorphisms.
It follows that $(\mathcal{C},\mathcal{D})$ and $(\mathcal{E},\mathcal{F})$ are both hereditary right $n$-cotorsion pairs by Lemma \ref{hereditary equiv}(1).

\noindent``$\Leftarrow$'' Since $(\mathcal{C},\mathcal{D})$ and $(\mathcal{E},\mathcal{F})$ are both right $n$-cotorsion pairs, then $\mathcal{D}$ and $\mathcal{F}$ are both closed under extensions. And then $\mathfrak{A}^\mathcal{D}_\mathcal{F}$ is closed under extensions.
Since $(\mathcal{C},\mathcal{D})$ and $(\mathcal{E},\mathcal{F})$ are both hereditary, then $\mathcal{D}$ and $\mathcal{F}$ are both closed under cokernels of monomorphisms by Lemma \ref{hereditary equiv}(1). This implies $\mathfrak{A}^\mathcal{D}_\mathcal{F}$ is closed under cokernels of monomorphisms.
So it remains to show $\Lambda\operatorname{-Inj}\subseteq\mathfrak{A}^\mathcal{D}_\mathcal{F}$.

For any injective module $\begin{pmatrix}
	M_1\\M_2
\end{pmatrix}_f$ in $\Lambda\operatorname{-Mod}$, it meets that $M_2\in B\operatorname{-Inj}\subseteq\mathcal{F}$, $\tilde{f}:M_1\to \operatorname{Hom}_B(U,M_2)$ is an epimorphism and $\operatorname{Ker}\tilde{f}\in A\operatorname{-Inj}\subseteq\mathcal{D}$ by Lemma \ref{projective and injective}.
Consider the short exact sequence:
$$0\to \operatorname{Ker}\tilde{f}\to M_1\to \operatorname{Hom}_B(U,M_2)\to 0.$$
For any $C\in\mathcal{C}$ and $1\leq i\leq n$, we have $\operatorname{Ext}^i_A(C,\operatorname{Hom}_B(U,M_2))\cong\operatorname{Hom}_B(\operatorname{Tor}^A_i(U,C),M_2)=0$ by \cite[Corollary 10.63]{JJRotman}.
So $\operatorname{Hom}_B(U,M_2)\in\bigcap\limits_{i=1}^n C^{\perp_i}=D$.
Since $\mathcal{D}$ is closed under extensions, then $M_1\in\mathcal{D}$.
Hence $\begin{pmatrix}
	M_1\\M_2
\end{pmatrix}_f\in\mathfrak{A}^\mathcal{D}_\mathcal{F}$. This shows $\Lambda\operatorname{-Inj}\subseteq\mathfrak{A}^\mathcal{D}_\mathcal{F}$ as desired.\\
(2) The proof is similar to (1).
\end{proof}

\begin{remark}
Under the condition $\mathfrak{A}^\mathcal{D}_\mathcal{F}$ (resp. $\mathfrak{A}^\mathcal{C}_\mathcal{E}$) is a coresolving class (resp. resolving class) in Theorem \ref{main result1}, we find out that all right (resp. left) $n$-cotorsion pairs in Theorem \ref{main result1} are actually hereditary.
\end{remark}
\begin{theorem}\label{main result2}
\noindent$(1)$ Let $(\mathcal{C},\mathcal{D})$ be a hereditary right $n$-cotorsion pair in $A\operatorname{-Mod}$ and $(\mathcal{E},\mathcal{F})$ be a hereditary right $n$-cotorsion pair in $B\operatorname{-Mod}$.
If $\operatorname{Tor}^A_j(U,\mathcal{C})=0$ for $1\leq j\leq n+1$ and $(\mathfrak{P}^\mathcal{C}_\mathcal{E})^\vee_{n-1}$ is
closed under extensions, then $(\mathfrak{P}^\mathcal{C}_\mathcal{E},\mathfrak{A}^\mathcal{D}_\mathcal{F})$ is a hereditary right $n$-cotorsion pair in $\Lambda\operatorname{-Mod}$.

\noindent$(2)$ Let $(\mathcal{C},\mathcal{D})$ be a hereditary left $n$-cotorsion pair in $A\operatorname{-Mod}$ and $(\mathcal{E},\mathcal{F})$ be a hereditary left $n$-cotorsion pair in $B\operatorname{-Mod}$.
If $\operatorname{Ext}_B^j(U,\mathcal{F})=0$ for $1\leq j\leq n+1$ and $(\mathfrak{I}^\mathcal{D}_\mathcal{F})^\wedge_{n-1}$ is closed under extensions, then $(\mathfrak{A}^\mathcal{C}_\mathcal{E},\mathfrak{I}^\mathcal{D}_\mathcal{F})$ is a hereditary left $n$-cotorsion pair in in $\Lambda\operatorname{-Mod}$.
\end{theorem}
\begin{proof}
By Theorem \ref{main result1}, Lemma \ref{hereditary equiv} and Lemma \ref{hereditary eq2}, we can obtain the results.
\end{proof}

Given a ring $R$, we denote by $R\operatorname{-mod}$
the category of finitely generated left $R$-modules and $R\operatorname{-proj}$ (resp. $R\operatorname{-inj}$) the class of finitely generated projective (resp. injective) left $R$-modules.
\begin{remark}\label{remark}
Suppose $A$, $B$ are finite dimensional algebras over an algebraically closed field $K$, and $U$ is a $(B,A)$-bimodule such that ${}_BU$ and $U_A$ are finitely generated, then the triangular matrix ring $\Lambda$ is also a finite dimensional algebra over $K$.
Further, keeping the assumption above, if we restrict all module categories ($A\operatorname{-Mod}$, $B\operatorname{-Mod}$ and $\Lambda\operatorname{-Mod}$) to finitely generated module categories ($A\operatorname{-mod}$, $B\operatorname{-mod}$ and $\Lambda\operatorname{-mod}$), it is easy to see that the main result Theorem \ref{main result2} still holds.
\end{remark}

\begin{example}\label{example}\rm
Let $Q$ be the quiver $3\to 2 \to 1$ and $\Lambda=KQ$. Then there exists a $K$-algebra isomorphism
$\Lambda\cong\begin{pmatrix}
   K & 0 & 0 \\
   K & K & 0 \\
   K & K & K
\end{pmatrix}$. So $\Lambda$ can be seen as a formal triangular matrix ring $\begin{pmatrix}
A & 0\\
{}_BU_A & B
\end{pmatrix}$, where $A=\begin{pmatrix}
K & 0\\
K & K
\end{pmatrix}$, $B=K$ and ${}_BU_A=(K,K)$ with the natural module structure.
And it is easy to see $U$ is a projective left $B$-module and projective right $A$-module.
The AR-quiver of $\Lambda\operatorname{-mod}$ can be represented as follows:
$$
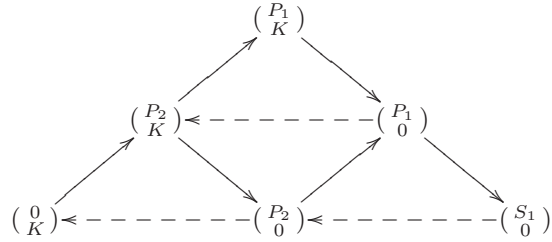
\xymatrix@1{ & & {(\begin{smallmatrix}
		P_1\\K
\end{smallmatrix})}\ar[rd] &  &  \\
    &  {(\begin{smallmatrix}
    	P_2\\K
    \end{smallmatrix})}\ar[ru]\ar[rd]  &   & {(\begin{smallmatrix}
    P_1\\0
\end{smallmatrix})}\ar@{-->}[ll]\ar[rd]   & \\
{(\begin{smallmatrix}
	0\\K
\end{smallmatrix})}\ar[ru] &  & {(\begin{smallmatrix}
P_2\\0
\end{smallmatrix})}\ar@{-->}[ll]\ar[ru] &  & {(\begin{smallmatrix}
S_1\\0
\end{smallmatrix})}\ar@{-->}[ll]
}$$

$(a)$ Let $(\mathcal{C},\mathcal{D})=(A\operatorname{-inj},A\operatorname{-inj})$ and $(\mathcal{E},\mathcal{F})=(B\operatorname{-mod},B\operatorname{-mod}) $.
It is easy to check that $(\mathcal{C},\mathcal{D})$ and $(\mathcal{E},\mathcal{F})$ are both hereditary right 2-cotorsion pairs.
By the constructions in Definition \ref{construction}, we can calculate that $\mathfrak{P}^\mathcal{C}_\mathcal{E}=\operatorname{add}\{(\begin{smallmatrix}
		0\\K
	\end{smallmatrix}),(\begin{smallmatrix}
	P_1\\K
	\end{smallmatrix}),(\begin{smallmatrix}
	S_1\\0
	\end{smallmatrix})\}$
and $\mathfrak{A}^\mathcal{D}_\mathcal{F}=\operatorname{add}\{(\begin{smallmatrix}
	0\\K
\end{smallmatrix}),(\begin{smallmatrix}
	P_1\\K
\end{smallmatrix}),(\begin{smallmatrix}
P_1\\0
\end{smallmatrix}),(\begin{smallmatrix}
	S_1\\0
\end{smallmatrix})\}$.
And $(\mathfrak{P}^\mathcal{C}_\mathcal{E})^\vee_1=\operatorname{add}\{(\begin{smallmatrix}
	0\\K
\end{smallmatrix}),(\begin{smallmatrix}
P_2\\K
\end{smallmatrix}),(\begin{smallmatrix}
	P_1\\K
\end{smallmatrix}),(\begin{smallmatrix}
	S_1\\0
\end{smallmatrix})\}$.
It is easy to check $(\mathfrak{P}^\mathcal{C}_\mathcal{E})^\vee_1$ is closed under extensions.
Since $U_A$ is projective as right $A$-module, then it satisfies that $\operatorname{Tor}^A_i(U,\mathcal{C})=0$ for $1\leq i\leq 3$.
By Theorem \ref{main result2} and Remark \ref{remark}, $(\mathfrak{P}^\mathcal{C}_\mathcal{E},\mathfrak{A}^\mathcal{D}_\mathcal{F})$ is a hereditary right 2-cotorsion pair in $\Lambda\operatorname{-mod}$.

$(b)$ Let $(\mathcal{C},\mathcal{D})=(A\operatorname{-proj},A\operatorname{-proj})$ and $(\mathcal{E},\mathcal{F})=(B\operatorname{-mod},B\operatorname{-mod})$.
It is clear that $(\mathcal{C},\mathcal{D})$ and $(\mathcal{E},\mathcal{F})$ are both hereditary left 2-cotorsion pairs.
By the constructions in Definition \ref{construction}, we can calculate that
$\mathfrak{A}^\mathcal{C}_\mathcal{E}=\operatorname{add}\{(\begin{smallmatrix}
	0\\K
\end{smallmatrix}),(\begin{smallmatrix}
	P_2\\K
\end{smallmatrix}),(\begin{smallmatrix}
P_2\\0
\end{smallmatrix}),(\begin{smallmatrix}
	P_1\\K
\end{smallmatrix}),(\begin{smallmatrix}
	P_1\\0
\end{smallmatrix})\}$ and $\mathfrak{I}^\mathcal{D}_\mathcal{F}=\operatorname{add}\{(\begin{smallmatrix}
P_1\\K
\end{smallmatrix}),(\begin{smallmatrix}
P_1\\0
\end{smallmatrix}),(\begin{smallmatrix}
P_2\\0
\end{smallmatrix})\}$.
And $(\mathfrak{I}^\mathcal{D}_\mathcal{F})^\wedge_1=\operatorname{add}\{(\begin{smallmatrix}
	P_1\\K
\end{smallmatrix}),(\begin{smallmatrix}
	P_1\\0
\end{smallmatrix}),(\begin{smallmatrix}
	P_2\\0
\end{smallmatrix}),(\begin{smallmatrix}
S_1\\0
\end{smallmatrix})$.
It is easy to check $(\mathfrak{I}^\mathcal{D}_\mathcal{F})^\wedge_1$ is closed under extensions.
Since ${}_BU$ is projective as left $B$-module, then it satisfies that $\operatorname{Ext}^i_B(U,\mathcal{F})=0$ for $1\leq i\leq 3$.
By Theorem \ref{main result2} and Remark \ref{remark}, $(\mathfrak{A}^\mathcal{C}_\mathcal{E},\mathfrak{I}^\mathcal{D}_\mathcal{F})$ is a hereditary left 2-cotorsion pair in $\Lambda\operatorname{-mod}$.
\end{example}

	\section*{Acknowledgments}
	This work is supported by the National Natural Science Foundation of China
	(Grant No. 11871145), the Jiangsu Provincial Scientific Research Center of Applied Mathematics (No. BK20233002) and the Qing Lan Project of Jiangsu Province, China.\\

%MS++++++++++++++++++++++++++++++ Reference ++++++++++++++++++
	
\end{document}